\documentclass[11pt,reqno]{amsart}

\numberwithin{equation}{section}
\usepackage{times}
\usepackage{amsmath,amsfonts,amstext,amssymb,amsbsy,amsopn,amsthm,eucal}
\usepackage{txfonts}
\usepackage{dsfont}
\usepackage{graphicx}   
\usepackage{hyperref}
\usepackage{accents}
\usepackage{enumerate}

\usepackage{color}

\newcommand{\diam}{{\rm diam}}

\newcommand{\Alex}{\text{Alex\,}}

\newcommand{\Alexnk}{\text{Alex}^n(\kappa)}

\newcommand{\dN}{\mathds{N}}

\newcommand{\dR}{\mathds{R}}
\newcommand{\dS}{\mathds{S}}

\newcommand{\cH}{\mathcal{H}}

\newcommand{\geod}[1]{[\,#1\,]}

\newcommand{\Rnkv}{\mathcal M(n,\kappa,v)}
\newcommand{\Lnkv}{\mathcal M_\infty(n,\kappa,v)}

\newcommand{\vol}[1]{\text{vol}\left(#1\right)}

\newcommand{\scup}[2]{\underset{#1}{\overset{#2}\cup}}
\newcommand{\dsju}[2]{\underset{#1}{\overset{#2}\amalg}}

\newcommand{\dsp}{\displaystyle}
\newcommand{\cnnt}[2]{{#2}^{*#1}}

\newtheorem{theorem}{Theorem}[section]

\newtheorem{property}[theorem]{Property}
\newtheorem{lemma}[theorem]{Lemma}
\newtheorem{corollary}[theorem]{Corollary}
\theoremstyle{definition}

\theoremstyle{remark}

\theoremstyle{remark}
\theoremstyle{example}
\newtheorem{example}{Example}[section]

\theoremstyle{remark}\newtheorem{conjecture}{Conjecture}[section]
\theoremstyle{remark}
\theoremstyle{remark}\newtheorem{problem}[conjecture]{Problem}

\begin{document}

\title{Lipschitz-Volume Rigidity and Globalization}

\author{Nan Li}

\thanks{The author was partially supported by the PSC-CUNY Research Award \#61533-00 49.}

\date{\today}
\maketitle

\begin{abstract}
  Let $X$ and $Y$ be length metric spaces. Let $\cH^n$ denote the $n$-dimensional Hausdorff measure. The Lipschitz-Volume Rigidity is a property that if there exists a 1-Lipschitz map $f\colon X\to Y$ and $0<\cH^n(X)=\cH^n(f(X))<\infty$, then $f$ preserves the length of path. This property holds for smooth manifolds but doesn't hold for all singular spaces. We survey the Lipschitz-Volume Rigidity Theorems on singular spaces with lower curvature bounds and discuss some related open problems.
\end{abstract}

%
%
%


\tableofcontents

\section{Introduction}\label{s:intro}

In this paper, we let $\dim_\cH$ denote the Hausdorff dimension and $\cH^n$ denote the $n$-dimensional Hausdorff measure. A map $f$ is called a path isometry if it preserves the length of path. Note that a local isometry is always a path isometry but the converse may not be true, since a path isometry is not necessary a one-to-one map. The Lipschitz-Volume Rigidity property, abbreviated as LV-Rigidity is stated as follows.

\begin{property}[LV-Rigidity]\label{p:lv-rig}
  Let $X$ and $Y$ be length metric spaces. If there exists a 1-Lipschitz map $f\colon X\to Y$ and $0<\cH^n(X)=\cH^n(f(X))<\infty$, then $f$ is a path isometry.
\end{property}

By a co-area formula type of argument, one can show that this property holds for smooth manifolds.

\begin{theorem}\label{t:LV-rig-mfd}
  The LV-Rigidity holds if $X$ and $Y$ are both closed Riemannian manifolds.
\end{theorem}

\begin{proof}
  Suppose $\dim_\cH(X)=\dim_\cH(Y)=n$. Let $L(\sigma)$ denote the length of a curve $\sigma$. Fix a geodesic $\gamma\subset X$. It's clear that $\mathcal L(\gamma)\ge \mathcal L(f(\gamma))$, since $f$ is 1-Lipschitz. For $r>0$ small, we have
  $$\vol {B_r(\gamma)}=c(n)\cdot r^{n-1}\mathcal L(\gamma)+O(r^n).$$
  Because $f(B_r(\gamma))\subseteq B_r(f(\gamma))$, we have
  $$\vol{f(B_r(\gamma))}\le \vol{B_r(f(\gamma))}\le c(n)\cdot r^{n-1}\mathcal L(f(\gamma))+O(r^n).$$
  Note that $\vol{B_r(\gamma)}=\vol{f(B_r(\gamma))}$. We have
  $$\mathcal L(\gamma)\le \mathcal L(f(\gamma))+O(r).$$
  Let $r\to 0$. We get $\mathcal L(\gamma)\le \mathcal L(f(\gamma))$.
\end{proof}

LV-Rigidity doesn't hold in general, even if $f$ is in addition a bi-Lipschitz map. Let ${\mathcal L}_d(\gamma)$ denote the length of curve $\gamma\subset X$ with respect to the metric $d$ on $X$.

\begin{example}[\cite{Li15-1}]\label{eg.cube.edge}
Given $n\ge 2$, let $(X,d)$ be an $n$-dimensional compact Riemannian manifold and $A\in X$ be a closed subset with $\cH^n(A)=0$, but $\cH^1(A)>0$. For example, $A$ can be a minimizing geodesic. Given $\lambda\in(0,1)$, let length metric $d_\lambda$ be induced by the following length structure:
$${\mathcal L}_{d_\lambda}(\sigma)
  \equiv\lambda\cdot\cH^1(\sigma\cap A)+\cH^1(\sigma\setminus A)
  ={\mathcal L}_d(\sigma)-(1-\lambda)\cdot \cH^1(\sigma\cap A).
$$
Here $\cH^1$ is the $1$-dimensional Hausdorff measure with respect to the metric $d\,|_{\,\sigma}$, re-normalized so that $\cH^1(\sigma)={\mathcal L}_d(\sigma)$. Let $f\colon (X,d)\to (X,d_\lambda)$ be the identity map. Note that $f$ is 1-Lipschitz, bi-Lipschitz onto and volume preserving. However, $f$ is not a path isometry.

If $\lambda=0$, the length structure ${\mathcal L}_{d_0}(\sigma)
\equiv\cH^1(\sigma\setminus A)$ induces a pseudometric $d_0$. By identifying the points with zero $d_0$-distance, i.e., the points in $A$, we obtain a length metric space $(X/d_0, \bar d_0)$. In this case, the projection map $f\colon X\to X/d_0$ is 1-Lipschitz onto and volume preserving, but not a bi-Lipschitz map or path isometry. It's also worth to point out that $X/d_0$ has only one singular point $p=f(A)$.
\end{example}

In the above example, $(Y,d_Y)=(X,d_\lambda)$ is a singular space. Note that $f$ is still a local isometry on the regular part $X\setminus A$. A main reason that the local isometry $f\,|_{\,X\setminus A}$ can't be extended to the entire $X$ is that not every geodesic in $(X, d_\lambda)$ can be  approximated by curves in $(X\setminus f(A), d_{X\setminus f(A)})$ with converging length. In particular, this shows that the metric completion of $(X\setminus f(A), d_{X\setminus f(A)})$ is not isometric to $(X, d_\lambda)$. In the case of $\lambda=0$, we have that Example \ref{eg.cube.edge} doesn't satisfy the following property.


\begin{property}\label{p:glob}
  Let $Y$ be a complete length metric space and $\dim_\cH(Y)=n$. Let $\mathcal U\subset Y$ be a subset with $\cH^{n-1}(Y\setminus \mathcal U)=0$, equipped with the intrinsic (length) metric $d_{\mathcal U}$. Then the metric completion of $(\mathcal U, d_{\mathcal U})$ is isometric to $(Y, d_Y)$.
\end{property}

For simplicity, let $X$ be a closed $2$-dimensional manifold and $A\subset X$ be a minimizing geodesic. If $\lambda=0$, then $p=f(A)$ is a single point. Consider $(Y,d_Y)=(X/d_0,\bar d_0)$ and $\mathcal U=Y\setminus \{p\}$. Note that the intrinsic metric on $\mathcal U$ is the same as the intrinsic metric on $X\setminus A$. Thus the metric completion of $(\mathcal U, d_{\mathcal U})$ is isometric to the metric completion of $(X\setminus A, d_{X\setminus A})$, which is not isometric to $(Y, d_Y)$ or $(X,d)$. 
Moreover, given $\lambda\in[0,1)$, let $\gamma=f(\sigma)\in Y$ be a minimizing geodesic passing through $f(A)$, where $\sigma\subset X$ is a curve with $\cH^1(\sigma\cap A)>0$. We claim that for any minimizing geodesics $\gamma_i\in Y$ that converges to $\gamma$, there is a subsequence $i\to\infty$ for which $\gamma_i\cap f(A)\neq\varnothing$. Arguing by contradiction, suppose $\gamma_i\cap f(A)=\varnothing$ for all $i\ge N$, then
$$\lim_{i\to \infty}\,{\mathcal L}_{d_\lambda}(\gamma_i)
=\lim_{i\to \infty}\,{\mathcal L}_d(\gamma_i)= {\mathcal L}_d(\gamma)
>{\mathcal L}_{d_\lambda}(\gamma).$$
This contradicts to $\dsp\lim_{i\to \infty}\,{\mathcal L}_{d_\lambda}(\gamma_i)={\mathcal L}_{d_\lambda}(\gamma)$, provided that $\gamma_i$ and $\gamma$ are all minimizing geodesics and $\gamma_i\to\gamma$. Therefore, the metric $d_\lambda$ on $Y$ can't be determined by the intrinsic metric over $\mathcal U=Y\setminus f(A)$. In the case $\lambda=0$, we have $p\in\gamma_i$ for a subsequence $i\to\infty$. Thus these $\gamma_i$ are all bifurcated geodesics.

The following is an obvious lemma. Let $\geod{xy}$ denote a geodesic connecting $x$ and $y$.

\begin{lemma}\label{l:metric_c}
  Let $(X,d)$ be a complete length metric space and $\mathcal U\subset X$ be a subset. If for every $x,y\in X$ and every $\epsilon>0$, there exist $x', y'\in X$ so that $|xx'|+|yy'|<\epsilon$ and $\geod{x' y'}\subseteq \mathcal U$, then the metric completion of $(\mathcal U, d_{\mathcal U})$ is isometric to $(X,d)$.
\end{lemma}

In Section \ref{sec:LV-rig}, we survey the results related to Property \ref{p:lv-rig} and Property \ref{p:glob}. In Section \ref{sec:glob}, we discuss some open questions related to the LV-Rigidity.

\section{Lipschitz-Volume Rigidity}\label{sec:LV-rig}

It turns out that Property \ref{p:glob} holds if the curvature on singular space $Y$ is bounded from below in some sense. In this paper, we mainly focus on the singular spaces with (synthetic) lower Ricci curvature bounds or (synthetic) lower sectional curvature bounds.

\subsection{Ricci limit spaces}

Let $\Rnkv$ be the collection of $n$-dimensional Riemannian manifolds $(M,p)$ with Ricci curvature bounded from below by $-(n-1)\kappa$ and $\cH^n(B_1(p))\ge v>0$. By Cheeger-Gromov Compactness Theorem, $\Rnkv$ is pre-compact in the pointed Gromov-Hausdorff topology. Let $\Lnkv$ be the closure of $\Rnkv$.


\begin{theorem}\label{t:m.c.Ric}
  Let subset $\mathcal U\subset X\in \Lnkv$, $n\ge 2$, be equipped with the intrinsic metric $d_{\mathcal U}$. If $\cH^{n-1}(X\setminus \mathcal U)=0$, then the metric completion of $(\mathcal U, d_{\mathcal U})$ is isometric to $(X, d_X)$.
\end{theorem}

If $\mathcal U$ is open, then the above theorem follows from Lemma \ref{l:metric_c} and the proof of Theorem 3.7 in \cite{CC00-II}. For the general case, it follows from Lemma \ref{l:metric_c} and the following lemma. Assume $M_i\to X$. A geodesic $\gamma\subset X$ is called a limit geodesic if there exists geodesics $\gamma_i\subset M_i$ for which $\gamma_i\to \gamma$. We let $\geod{pq}_\infty$ denote the collection of limit geodesics connecting $p$ and $q$. Given $p\in X$ and $E\subseteq X$, define $H_R(p,E)=\{y\in \bar B_R(p)\colon (\geod{py}_\infty\cap E)\setminus\{p\}\neq\varnothing\}$.

\begin{lemma}
  Let $p\in X\in \Lnkv$ and $E\subseteq X$. If $\cH^{n-1}(E)=0$, then $\cH^n(H_R(p,E))=0$.
\end{lemma}

\begin{proof}
  Not losing generality, we assume $R=1$. We will need the following Bishop-Gromov volume comparison on the limit space $X$. Given $p,x\in X$ with $0<d(p,x)=r_0<1$ and $0<r<\min\{r_0, 1-r_0\}/100$, let $A_r(p,x)=\{y\in B_1(p)\colon\geod{py}_\infty\cap B_r(x)\neq\varnothing\}$.
Then
\begin{align}
  \mathcal H^n(A_r(p,x))\le c(n,\kappa,r_0)r^{n-1}.
\end{align}
This follows from Bishop-Gromov comparison on manifolds and the volume convergence theorem in \cite{CC97-I}.


Now for any $\epsilon>0$, cover $E_\epsilon=E\cap (B_{1-\epsilon}(p)\setminus B_\epsilon(p))$ by countably many balls $\{B_{r_\alpha}(x_\alpha)\}$ for which $\sum r_\alpha^{n-1}<\epsilon$ and $r_\alpha<\epsilon\le\min\{d(p,x_\alpha), 1-d(p,x_\alpha)\}/100$. Note that $H_1(p,E_\epsilon)\subseteq\cup_\alpha A_{r_\alpha}(p,x_\alpha)$. We have
$$\cH^n(H_1(p,E_\epsilon))\le \sum_\alpha \mathcal H^n(A_{r_\alpha}(p,x_\alpha))\le \sum_\alpha c(n,\kappa,r_0)r_\alpha^{n-1}<c(n,\kappa,r_0)\epsilon.$$
Therefore,
$$\cH^n(H_1(p,E))\le\cH^n(H_1(p,E_\epsilon))+\cH^n(\bar B_\epsilon(p))+\cH^n(\bar B_1(p)\setminus B_{1-\epsilon}(p))\le c(n,\kappa,r_0)\epsilon.$$
\end{proof}

For LV-Rigidity, we have the following theorem.
\begin{theorem}[\cite{LW14}]\label{t:LW-rig}
  Let $X, Y\in \Lnkv$ and $n\ge 2$. Suppose that there is a 1-Lipschitz map $f\colon X\to Y$. If $\cH^n(X)=\cH^n(f(X))$, then $f$ is an isometry with respect to the intrinsic metrics of $X$ and $f(X)$. In particular, if $f$ is also onto, then $Y$ is isometric to $X$.
\end{theorem}

\begin{corollary}[\cite{LW14}]\label{t:LW-a.rig}
  For any $n\ge 2$, $\kappa,v,D>0$ and $\epsilon>0$, there exists $\delta=\delta(n,\kappa,v,D,\epsilon)>0$ such that the following holds for any $X,Y\in \Lnkv$ that satisfies $\max\{\diam(X), \diam (Y)\}\le D$ and $|\cH^n(X)-\cH^n(Y)|<\delta$. If map $f:X\to Y$ satisfies $|f(x)f(y)|_Y\le |xy|_X+\delta$, for all $x,y\in X$, then $|f(x)f(y)|_Y\ge |xy|_X-\epsilon$ for any $x,y\in X$. Consequently, $f$ is an $\epsilon$-Gromov-Hausdorff approximation.
\end{corollary}

A special case of Corollary \ref{t:LW-a.rig} was proved by Bessi\`{e}res, Besson, Courtois, and Gallot in \cite{BBGG}. They used it to prove the following stability theorem.

\begin{theorem}[\cite{BBGG}]\label{t:BBGG}
  For any integer $n\ge 3$ and $D>0$, there exists $\epsilon(n,D)>0$ so that the following holds. Let $(Y,g)\in \mathcal M(n,-1,v)$ and $(X,g_0)$ be an $n$-dimensional closed hyperbolic manifold with $\diam(X)\le D$. Suppose that there exists a degree-one map $f\colon Y\to X$. Then $vol_g(Y)\le (1+\epsilon)vol_{g_0}(X)$ if and only if $f$ if homotopic to a diffeomorphism.
\end{theorem}

\subsection{Alexandrov spaces}

Let $\Alex^n(\kappa)$ denote the collection of $n$-dimensional Alexandrov spaces with curvature $\ge\kappa$. For $X\in\Alex^n(\kappa)$ the Toponogov comparison holds in the sense that the geodesic triangles are ``fatter" than the corresponding triangles in the space form $\dS^2_\kappa$.  For example, the Gromov-Hausdorff limits of Riemannian manifolds with $\sec\ge\kappa$ are Alexandrov spaces with curvature $\ge \kappa$. The quotient space $M/G\in\Alex(\kappa)$, where $M$ is a Riemannian manifold with $sec_{M}\ge\kappa$ and group $G$ acts on $M$ isometrically.

To compare with Ricci limit spaces, we would like to point out that not every Alexandrov space is isometric to a non-collapsed limit of Riemannian manifolds with uniform lower sectional curvature bound, due to some topological obstruction (c.f. \cite{Kap07}). It is an open question whether every Alexandrov space is a collapsed limit of Riemannian manifolds with uniform lower sectional curvature bound.

Similar to Theorem \ref{t:m.c.Ric}, we have the following result.

\begin{theorem}\label{t:met.comp.Alex}
  Let $n\ge 2$ and subset $\mathcal U\subset X\in \Alex^n(\kappa)$ be equipped with the intrinsic metric $d_{\mathcal U}$. If $\dim_{\cH}(X\setminus \mathcal U)< n-1$, then the metric completion of $(\mathcal U, d_{\mathcal U})$ is isometric to $(X, d_X)$.
\end{theorem}

This follows from Lemma \ref{l:metric_c} and the following lemma for which we will outline the proof.

\begin{lemma}[Dimension comparison]\label{map.dim}
  Let $\Omega_0\subseteq X\in\Alexnk$ be a subset and $p\in X$ be a fixed point. Let $\Omega\subseteq X$ be a subset such that $d(p,\Omega)>0$. If for each $x\in \Omega_0$ there exists a geodesic $\geod{px}$ such that $\Omega\cap\geod{px}\neq\varnothing$,
  then
  $$\dim_\cH(\Omega)\ge\dim_\cH(\Omega_0)-1.$$
\end{lemma}

\begin{proof}
  Let $\Gamma=\Omega\times[0,\infty)$ be equipped with the metric
  $$d((x_1,t_1),(x_2,t_2))=|x_1x_2|_X+|t_1-t_2|,$$
  where $x_i\in\Omega$ and $t_i\in[0,\infty)$, $i=1,2$.
  Define map $h: \Omega_0\to\Gamma$, $x\mapsto (\bar x,|px|_X)$, where $\geod{px}$ is selected such that $\geod{px}\cap\Omega\neq\varnothing$ and $\bar x\in \geod{px}\cap\Omega$ is selected arbitrarily. By Toponogov's Comparison Theorem, one can show that there is a constant $c>0$ so that for any $x_1,x_2\in \Omega_0$, we have
  \begin{align}
    |h(x_1)h(x_2)|_{\Gamma}\ge c\cdot|x_1x_2|_X.
    \label{map.dim.e1}
  \end{align}
  Therefore, we have
  $$\dim_\cH(\Omega)+1\ge\dim_\cH(\Gamma)\ge \dim_\cH(\Omega_0).$$
\end{proof}

Let us begin with some special cases of LV-Rigidity. Let $X\in\Alex^n(1)$ and $p\in X$. Because Topogonov comparison holds on $X$ with respect to the space form $\mathds S^2_1$, we naturally have a distance non-decreasing map $h\colon X\to\mathds S^n_1$. If $\cH^n(X)=\cH^n(\mathds S^n_1)$, one can find a 1-Lipschitz onto map $f\colon \mathds S^n_1 \to X$ for which $\cH^n(\mathds S^n_1)=\cH^n(f(\mathds S^n_1))=\cH^n(X)$. This is in turn a type of LV-Rigidity problem.

\begin{theorem}[\cite{BGP}]
  If $X\in\Alex^n(1)$ and $\cH^n(X)=\cH^n(\mathds S^n_1)$, then $X$ is isometric to $\mathds S^n_1$.
\end{theorem}

This was generalized to the so-called $\kappa$-tangent cone rigidity by N. Li and X. Rong. Let $C^R_\kappa(\Sigma)$ be the metric cone over $\Sigma_p$, on which the Euclidean cosine law is replaced by the cosine law on the space form $S^2_\kappa$. In particular, if $\kappa>0$, then $C^R_\kappa(\Sigma)$ is the $\kappa$-suspension of $\Sigma$. Given $\Sigma\in\Alex^{n-1}(1)$, $n\in\dN$, $\kappa\in\dR$ and $R>0$, let $\mathcal A^R_{n,\,\kappa}(\Sigma)$
denote the collection of Alexandrov spaces $(X,p)\in\Alex^n(\kappa)$ so that $X\subseteq\bar B_R(p)$ and there exists a 1-Lipschitz map $\varphi\colon \Sigma\to\Sigma_p$. For any $(X,p)\in \mathcal A^R_{n,\,\kappa}(\Sigma)$, we have
$$\cH^n(X)\le\cH^n(B_R(p))\le \cH^n(C^R_\kappa(\Sigma))=v(\Sigma,n,\kappa,R).$$

Now fix $\Sigma\in\Alex^{n-1}(1)$, $n\in\dN$, $\kappa\in\dR$ and $R>0$. Given any isometric involution $\phi\colon\Sigma\to\Sigma$, the quotient space $\bar C^R_\kappa(\Sigma)
/((x,R)\sim (\phi(x),R))\in \mathcal A^R_{n,\,\kappa}(\Sigma)$, whose $n$-dimensional Hausdorff measure is equal to $v(\Sigma,n,\kappa,R)$. This gives us some model spaces in $\mathcal A^R_{n,\,\kappa}(\Sigma)$ which have the maximum volume.

\begin{theorem}[{Relatively maximum volume rigidity} \cite{LR12}]\label{t:rmv-rig}
Let $X\in\mathcal A^R_{n,\,\kappa}(\Sigma)$ such that $\cH^n(X)=\cH^n(C^R_\kappa(\Sigma))$.
Then $X$ is isometric to $\bar C^R_\kappa(\Sigma)/((x,R)\sim (\phi(x),R))$, where $\phi: \Sigma\to\Sigma$ is an isometric involution. Moreover, if $\kappa>0$, then either $R\le \frac \pi{2\sqrt \kappa}$ or $R=\frac \pi{\sqrt \kappa}$.
\end{theorem}

As an outline of the proof, the first step is to construct a 1-Lipschitz onto map $f\colon \bar C^R_\kappa(\Sigma)\to X$. Then an LV-Rigity theorem was proved for this special case. The map $f$ may not be a local isometry since $\phi$ can be a non-trivial involution. If the fixed point set of $\phi$ is of dimension smaller than $n-2=\dim(\Sigma)-1$, then $X$ may not be a topological manifold. The topological types of $X$ can be classified if it is known to be a topological manifold.

\begin{theorem}[\cite{Li15-1},\cite{LR12}]
Let the assumptions be the same as in Theorem \ref{t:rmv-rig}. If $X$ is a topological manifold without boundary, then $X$ is either homeomorphic to $\Bbb S^n$ or homotopy equivalent to $\Bbb{RP}^n$. In particular, $\dim(\text{Fix}_\phi)=n-2=\dim(\Sigma)-1$.
\end{theorem}

This is also the counterpart of K. Grove and P. Petersen's results in Riemannian Geometry in \cite{GP92}.

Note that in Theorem \ref{t:rmv-rig}, $f$ is a local isometry restricted to the interior of $\bar C^R_\kappa(\Sigma)$ and there is possibly a gluing structure over the boundary, induced by an isometric involution $\phi$. It was asked whether this kind of boundary gluing structure still holds for general Alexandrov spaces, provided a 1-Lipschitz, volume preserving onto map $f\colon X\to Y$. This was confirmed in \cite{Li15-1}.

By $(\amalg_\ell X_\ell, d)$ we denote the disjoint union of length metric spaces $\{(X_\ell,d_\ell)\}$, where $d(p,q)=d_\ell(p,q)$ if $p,q\in X_\ell$ for some $\ell$ and $d(p,q)=\infty$ otherwise. Let $ X^\circ=\amalg_\ell X_\ell^\circ$ denote the interior of $X$ and $\partial X=\amalg_\ell\partial X_\ell$ be the boundary.

\begin{theorem}[\cite{Li15-1}]\label{t:Li15-1-1-lip}
  Let $Y\in\Alex^n(\kappa)$ and $X=\dsju{\ell=1}{N_0} X_\ell$, where $X_\ell\in\Alex^n(\kappa)$ for all $1\le \ell\le N_0$. If a 1-Lipschitz map $\dsp f\colon X\to Y$ satisfies $\cH^n(X)=\cH^n(f(X))$, then $f$ is a path isometry and $f\,|_{\, X^\circ}$ is an isometry with respect to the intrinsic metrics.
\end{theorem}

Let $\mathcal R$ be an equivalence relation on $(X,d)$. The quotient pseudometric $d_{\mathcal{R}}$ on $X$ is defined as
$$d_{\mathcal{R}}(p,q)
  =\inf\left\{\sum_{i=1}^Nd(p_i,q_i):p_1=p, q_N=q, p_{i+1}\overset {\mathcal{R}}\sim q_i,
  N\in\dN
\right\}.$$
See $\S3$ in \cite{BBI} for more details. By identifying the points with zero $d_{\mathcal{R}}$-distance, one obtains a length metric space $(X/d_{\mathcal{R}}, \bar d_{\mathcal{R}})$, which is called the space glued from $\{X_\ell\}$ along the equivalence relation $d_{\mathcal{R}}$. The projection map $f\colon X\to X/d_{\mathcal{R}}$ is always a 1-Lipschitz onto.

In Theorem \ref{t:Li15-1-1-lip}, we can define an equivalence relation $\mathcal R$ on $(X,d)$ by $x_1\sim x_2$ $\Leftrightarrow$ $f(x_1)=f(x_2)$. It's clear that if $f$ is onto, then $Y$ is isometric to the glued space $(X/d_{\mathcal{R}}, \bar d_{\mathcal{R}})$. If $f$ is a path isometry, then any two glued paths have the same length. In fact, suppose that $\gamma_1, \gamma_2\subset X$ are glued. That is, $f(\gamma_1(t))=f(\gamma_2(t))$ for all $t$. Then $$\mathcal L_{d_X}(\gamma_1)=\mathcal L_{d_Y}(f(\gamma_1))=\mathcal L_{d_Y}(f(\gamma_2))=\mathcal L_{d_X}(\gamma_2).$$
Such a gluing structure is called gluing by isometry.

To guarantee $Y\in\Alexnk$, there are more restrictions on the way how $X_\ell\in\Alex^n(\kappa)$ are glued. Let $G_Y=\{y\in Y\colon |f^{-1}(y)|>1\}$ and $G_X=f^{-1}(G_Y)$. There is a natural stratification of the glued and gluing points. Namely,
$$G_Y^m=\left\{y\in Y\colon\; \left|f^{-1}(y)\right|=m\right\}
  \text{\quad and \quad} G_X^m=f^{-1}(G_Y^m).
$$

\begin{theorem}[\cite{Li15-1}]\label{t:Li15-1-1-onto}
  Let the assumptions be the same as in Theorem \ref{t:Li15-1-1-lip}. If $f$ is onto, then the following hold.
  \begin{enumerate}
    \renewcommand{\labelenumi}{(\arabic{enumi})}
      \setlength{\itemsep}{1pt}
      \item $G_X\subseteq\partial X$.
      \item For any $y\in Y$,  we have $\big|f^{-1}(y)\big|<\infty$. Moreover, $$\max\{m:G_Y^m\neq\varnothing\}\le C(n,\kappa,d_0,v_0),$$
          where $d_0=\underset{1\le\ell\le N_0}\max\{\text{diam}(X_\ell)\}$ and $v_0=\underset{1\le\ell\le N_0}\min\{\vol{X_\ell}\}$.
      \item If $G_X\neq\varnothing$, then for any point ${x}\in G_X$ and $r>0$, we have
        $$\dim_\cH\Big(G_X^2 \cap B_r({x})\Big)=\dim_\cH\Big(G_Y^2 \cap B_r(f(x))\Big)=n-1.$$
      \item $$\dim_\cH\left(\scup{m=3}{\infty}G_X^m\right) =\dim_\cH\left(\scup{m=3}{\infty}G_Y^m\right)
        \le n-2.$$
\end{enumerate}
\end{theorem}

 By (3) and (4), the gluing structure is uniquely determined by the way how $G_X^2$ is glued. The gluing in the case $N_0=1$ (for example, Theorem \ref{t:rmv-rig}) is called {\it self-gluing}. In fact, Theorem \ref{t:Li15-1-1-onto} shows that without losing volume or increasing the distance, the metric on a connected Alexandrov space is ``rigid'' up to a boundary gluing by isometry.

The map $f$ is an isometry in some special cases.

\begin{corollary}[\cite{Li15-1}]\label{shrinking.cor}
  Under the same assumptions as in Theorem \ref{t:Li15-1-1-onto}, if any of the following is satisfied then $N_0=1$ and $f$ is an isometry.
  \begin{enumerate}
  \item $\partial X_\ell=\varnothing$ for some $\ell$.
  \item $G_X=\varnothing$, i.e., $f$ is injective.
  \item $G_Y\subseteq\partial Y$.
  \item $f(\partial X)\subseteq \partial Y$.
  \end{enumerate}
\end{corollary}

In the last of this section, we discuss some applications of the Lipschitz-Rigidity Theorem in ALexandrov Geometry. First, Theorem \ref{t:Li15-1-1-onto} implies Theorem \ref{t:rmv-rig} with some extra work on the involution part. 
For applications on gluing, let's recall the following theorem.

%
%

\begin{theorem}[Petrunin, \cite{Pet97}]\label{pet.glu}
The gluing of $X_1, X_2\in\Alex^n(\kappa)$ via an isometry $\phi\colon \partial X_1\to\partial X_2$ produces an Alexandrov space with the same lower curvature bound.
\end{theorem}


The following result, conjectured by A. Petrunin, follows from Theorem \ref{t:Li15-1-1-onto} and Petrunin's Gluing Theorem.


\begin{theorem}[\cite{Li15-1}]\label{pet.iff}
  Assume that $X_1, X_2\in\Alex^n(\kappa)$ are glued via an identification $x\sim \phi(x)$, where $\phi:\partial X_1\to\partial X_2$ is a bijection. Then the glued space $Y=X_1\amalg X_2/(x\sim \phi(x))$ is an Alexandrov space if and only if $\phi$ is an isometry with respect to the intrinsic metrics of $\partial X_1$ and $\partial X_2$.
\end{theorem}

Let $p\in X\in\Alexnk$ and $\Sigma_p$ denote the space of directions of $X$ at $p$. It is known that there is a distance non-decreasing map $\dsp h\colon \Sigma_p\to\lim_{i\to\infty}\Sigma_{p_i}$, provided $p_i\to p$ and $\dsp\lim_{i\to\infty}\Sigma_{p_i}$ exists (Theorem 7.14, \cite{BGP}). However, it may happen that $\dsp\lim_{i\to\infty}\Sigma_{p_i}\neq\Sigma_p$. For example, regular points can converge to a singular point. 
In \cite{Pet98}, it was proved that $\Sigma_q$ is isometric to $\Sigma_p$ if they are interior points along a fixed geodesic. As an application of Corollary \ref{shrinking.cor} (3), we have the following theorem.





\begin{theorem}[LV-rigidity of spaces of directions, \cite{Li15-1}]\label{shrink.spd}
  Let $X_i\in\Alexnk$. Suppose that $(X_i, p_i)\overset{d_{GH}}{\longrightarrow}(X, p)$. Then $\dsp\lim_{i\to\infty}\Sigma_{p_i}=\Sigma_p$ if and only if  $\dsp\lim_{i\to\infty}\cH^{n-1}\left(\Sigma_{p_i}\right)=\cH^{n-1}\left(\Sigma_p\right)$.
\end{theorem}

\section{Globalization in ALexandrov Geometry}\label{sec:glob}

Let us start with Petrunin's Gluing Theorem (Theorem \ref{pet.glu}). Using the notion of Theorem \ref{t:Li15-1-1-onto}, Petrunin's Theorem requires that $N_0=2$ and  $G_X=G_X^2=\partial X=\partial X_1\amalg\partial X_2$. One can ask whether such a gluing theorem can be generalized to the gluing of more than two Alexandrov spaces and $G_X\neq G_X^2$. Note that Theorem \ref{t:Li15-1-1-onto} in fact provides some of the necessary conditions. However, these conditions are not sufficient.

\begin{example}\label{eg.non.extremal} Let $\triangle A_1B_1C_1$ and $\triangle A_2B_2C_2$ be two triangle planes in $\dR^2$, with $|B_1C_1|=|B_2C_2|$. Glue the two triangles via identification $B_1C_1\sim B_2C_2$. The resulted space $Y\in\Alex^2(0)$ if and only if
$$\max\Big\{\measuredangle A_1B_1C_1+\measuredangle A_2B_2C_2, \, \measuredangle A_1C_1B_1+\measuredangle A_2C_2B_2\Big\}\le\pi.$$
If the above inequality is not satisfied, then $Y$ is strictly concave. Thus it is not an Alexandrov space.

\begin{center}\includegraphics[scale=1]{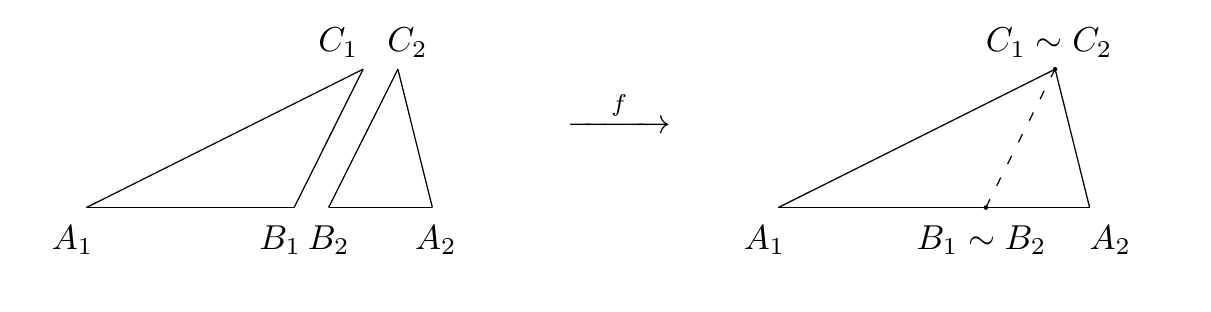}\end{center}

\end{example}

Based on this observation, we have the following conjecture.

\begin{conjecture}[\cite{Li15-1}]\label{conj:g.Alex}
  A gluing along boundaries of $n$-dimensional Alexandrov spaces produces an Alexandrov space $Y$ if and only if the gluing is by isometry and for each $p\in Y$, the tangent cone $\dsp T_p(Y)=\lim_{r\to 0}(Y, p, r^{-1}d)$ is a metric cone over $\Sigma_p$, where $\Sigma_p\in\Alex^{n-1}(1)$.
\end{conjecture}

If the gluing only happens along the boundaries, then for most points, such as interior points, there exist neighborhoods in which Toponogov comparison holds. In this sense, the above conjecture might be related to the following Globalization Problem.

Let $\mathcal U$ be a length metric space, not necessarily complete. An open set $U\subseteq \mathcal U$ is said to be a $\kappa$-Alexandrov domain, if $\kappa$-Toponogov comparison holds for any geodesic triangle in $U$. A length metric space $\mathcal U$ is said to be a local $\kappa$-Alexandrov space, if for any $p\in \mathcal U$ there is a $\kappa$-Alexandrov domain $U_p\ni p$. Let $\Alex^n_{loc}(\kappa)$ be the collection of all $n$-dimensional local $\kappa$-Alexandrov spaces. Note that local Alexandrov spaces are not necessarily (global) Alexandrov spaces if it is incomplete. For example, any open domain $\mathcal U\subset\dR^n$ is a local Alexandrov space, but the closure $\bar{\mathcal U}$ is not Alexandrov if $\mathcal U$ is strictly concave.

\begin{problem}[Globalization]\label{conj:glob}
  Let $(\mathcal U, d_{\mathcal U})\in\Alex^n_{loc}(\kappa)$. Under what condition is the metric completion $\overline{(\mathcal U, d_{\mathcal U})}\in \Alex^n(\kappa)$?
\end{problem}

Establishing a Globalization Theorem like this might help us to prove a space to be an Alexandrov space. For example, to prove $X\in \Alex^n(\kappa)$, one can take $\mathcal U\subset X$ with $\cH^{n-1}(X\setminus \mathcal U)=0$ and carry out the following three steps.
\begin{enumerate}
  \renewcommand{\labelenumi}{(\arabic{enumi})}
  \setlength{\itemsep}{1pt}
  \item Prove $(\mathcal U, d_{\mathcal U})\in\Alex^n_{loc}(\kappa)$.
  \item Prove that the metric completion $\overline{(\mathcal U, d_{\mathcal U})}\in \Alex^n(\kappa)$ by some Globalization Theorems.
  \item Prove that $\overline{(\mathcal U, d_{\mathcal U})}$ is isometric to $(X, d_X)$. Note that this doesn't follow from Theorem \ref{t:met.comp.Alex} directly, since $X$ is not known to be an Alexandrov space.
\end{enumerate}

Note that $X\in\Alex^n(\kappa)$ requires that triangle comparison holds for any size of triangles in $X$. Thus Problem \ref{conj:glob} is even not trivial when $\mathcal U=X$. This turns out to be a fundamental theorem in ALexandrov Geometry.

\begin{theorem}[\cite{BGP}]
  Let $(\mathcal U, d_{\mathcal U})\in\Alex^n_{loc}(\kappa)$. If $(\mathcal U, d_{\mathcal U})$ is complete, then $(\mathcal U, d_{\mathcal U})=\overline{(\mathcal U, d_{\mathcal U})}\in \Alex^n(\kappa)$.
\end{theorem}

This was generalized by A. Petrunin to the following result.

\begin{theorem}[\cite{Pet16}]\label{t:Pet.glob}
  Let $(\mathcal U, d_{\mathcal U})\in\Alex^n_{loc}(\kappa)$. If $(\mathcal U, d_{\mathcal U})$ is totally geodesic, then $\overline{(\mathcal U, d_{\mathcal U})}\in \Alex^n(\kappa)$.
\end{theorem}

A basic example for the above result is that $\mathcal U$ is an open convex domain in $\dR^n$, $n\ge2$. In this case, if we take $\mathcal V=\mathcal U\setminus \{A\}$, where $A\subset \mathcal U$ and $\dim_\cH(A)\le n-2$, then $\overline{(\mathcal V, d_{\mathcal V})}=\overline{(\mathcal U, d_{\mathcal U})}\in \Alex^n(0)$. However, in general, we can't simply conclude $\overline{(\mathcal V, d_{\mathcal V})}=\overline{(\mathcal U, d_{\mathcal U})}$ and use Theorem \ref{t:Pet.glob} directly. This is in particular because we don't know wether points in $A$ admit Alexandrov domains in general. See Example \ref{eg.cube.edge} and the discussion below Property \ref{p:glob}. Nevertheless, we would like to have a theorem to handle this case.

Let $\mathcal U$ be a locally compact length metric space. Given a point $p\in\mathcal U$ and a subset $S\subseteq\mathcal U$, let
$$\cnnt{p}{S}=\Big\{q\in S: \text{there is a geodesic } \geod{pq}\subseteq\mathcal U \text{ connecting } p \text{ and } q \Big\}.$$
Some classical convexities can be rephrased as follows using this terminology. Assume $\dim_\cH(\mathcal U)=n<\infty$
\begin{itemize}
\setlength{\itemsep}{1pt}
  \item Convex:\quad
    $\geod{pq}\subseteq\mathcal U$ for every $p,q\in\mathcal U$.
    \quad$\Longleftrightarrow$\quad
    $\cnnt{p}{\mathcal U}=\mathcal U$ for every $p\in\mathcal U$.
  \item A.e.-convex:\quad
    $\geod{pq}\subseteq\mathcal U$ for almost every $p,q\in\mathcal U$.
    \\ \hphantom \qquad \quad$\Longleftrightarrow$\quad
    $\mathcal H^n(\mathcal U\setminus\cnnt{p}{\mathcal U})=0$ for almost every point $p\in\mathcal U$.
  \item Weakly a.e.-convex:\quad
    For every $p\in\mathcal U$, there exists $p_i\to p$ such that $\geod{p_iq}\subset\mathcal U$ for a.e. $q\in \mathcal U$.
    \\ \hphantom \qquad \quad$\Longleftrightarrow$\quad
    For every $p\in\mathcal U$, there exists $p_i\to p$ such that $\mathcal H^n(\mathcal U\setminus\cnnt{p_i}{\mathcal U})=0$.
  \item Weakly convex:\quad
    For every $p,q\in\mathcal U$, there exists $p_i\to p$ and $q_i\to q$ such that $\geod{p_iq_i}\subset\mathcal U$.
    \\ \hphantom \qquad\quad$\Longleftrightarrow$\quad
    For every $p, q\in \mathcal U$ and any $\epsilon>0$, there exists $p_1\in B_\epsilon(p)$ such that $\cnnt{p_1}{B_\epsilon(q)}\neq\varnothing$.
\end{itemize}
The following notion of probabilistic convexity was introduced in \cite{Li15-2}. Let $p\in\mathcal U$ and $\geod{qs}$ be a geodesic in $\mathcal U$. Consider the probability that a point on $\geod{qs}$ can be connected to $p$ by a geodesic in $\mathcal U$:
$${\bf Pr}\left(\geod{qs}^{*p}\right)=\frac{\mathcal
  H^1\left(\cnnt p{\geod{qs}}\right)} {\mathcal
  H^1\left(\geod{qs}\right)}.
$$
Here $\mathcal H^1$ denotes the $1$-dimensional Hausdorff measure. We say that $\mathcal U$ is weakly $\mathfrak p_\lambda$-convex if for any $p,q,s\in\mathcal U$ and any $\epsilon>0$, there are points $p_1\in B_\epsilon(p)$, $q_1\in B_\epsilon(q)$, $s_1\in B_\epsilon(s)$ and a geodesic $\geod{q_1s_1}\subset\bar{\mathcal U}$ so that ${\bf Pr}(\geod{q_1 s_1}^{*p_1})>\lambda-\epsilon$. By choosing the point $s\in B_\epsilon(q)$, we see that if $\lambda>0$, then
weak $\mathfrak p_\lambda$-convexity implies weak convexity. If $\mathcal U\in\Alex_{loc}(\kappa)$, then weak a.e.-convexity implies weak $\mathfrak p_1$-convexity (c.f. \cite{Li15-2}).

\begin{theorem}[\cite{Li15-2}]\label{t:Li15-2-p1}
  If $\mathcal U\in\Alex_{loc}(\kappa)$ is weakly $\mathfrak p_1$-convex, then its metric completion $\bar{\mathcal U}\in\Alex(\kappa)$.
\end{theorem}

\begin{corollary}[\cite{Li15-2}]\label{cor.t:Li15-2-p1}
If $\mathcal U\in\Alex_{loc}(\kappa)$ is weakly a.e.-convex then its metric completion $\bar{\mathcal U}\in\Alex(\kappa)$.
\end{corollary}

If $\bar{\mathcal U}$ is known to be an Alexandrov space, then its optimal lower curvature bound is rigid, provided $\lambda>0$.

\begin{theorem}[\cite{Li15-2}]\label{t:Li15-2-p2}
  Suppose that $\mathcal U\in\Alex_{loc}(\kappa)$ is weakly $\mathfrak p_\lambda$-convex for some $\lambda>0$. If $\bar{\mathcal U}\in\Alex(\kappa_1)$ for some $\kappa_1>-\infty$, then $\bar{\mathcal U}\in\Alex(\kappa)$.
\end{theorem}

 The answers for the following two questions remain unknown to the author.
 \begin{itemize}
   \item{\it Is Theorem \ref{t:Li15-2-p1} still true if $\mathcal U$ is weakly $\mathfrak p_\lambda$-convex for some $\lambda\in(0,1)$?}
   \item {\it Is Theorem \ref{t:Li15-2-p2} still true if $\mathfrak p_\lambda$-convexity is replaced by weak convexity?}
 \end{itemize}

Inspired by Conjecture \ref{conj:g.Alex} and the Globalization Problem, we have the following conjecture.

\begin{conjecture}\label{conj:glob.cone}
  Let $(X,d_X)$ be a length metric space and $\mathcal U\subset X$ be an open subset with $\cH^{n-1}(X\setminus \mathcal U)=0$. Assume that $(\mathcal U, d_{\mathcal U})\in\Alex^n_{loc}(\kappa)$ and for every $p\in X$, the tangent cone $\dsp T_pX=\lim_{r\to 0}(X,p,r^{-1}d)$ exists and isometric to a metric cone $C(\Sigma_p)$, where $\Sigma_p\in\Alex^{n-1}(1)$. Then $(X,d_X)\in\Alex^n(\kappa)$.
\end{conjecture}

This conjecture is true for polytopes in Euclidean spaces. This is because the convexity of tangent cone at a point $p$ implies that $B_r(p)$ is convex for $r>0$ small. Then the result follows from that local convexity implies global convexity. More generally, we have the following result.

\begin{theorem}[\cite{LiGe19}]
  Conjecture \ref{conj:glob.cone} is true if $X\setminus \mathcal U$ is a discrete set.
\end{theorem}

\bibliographystyle{amsalpha}

\end{document}